\documentclass[12pt]{amsart}

\usepackage[margin=1in]{geometry}  
\usepackage{graphicx}              
\usepackage{amsmath}               
\usepackage{amsfonts}              
\usepackage{amsthm}                
\usepackage{amssymb}
\usepackage{amscd}
\usepackage{listings}
\usepackage{tikz}
\usetikzlibrary{arrows}

\newtheorem{thm}{Theorem}[section]
\newtheorem{lem}[thm]{Lemma}
\newtheorem{prop}[thm]{Proposition}
\newtheorem{cor}[thm]{Corollary}

\theoremstyle{definition}
\newtheorem{defn}[thm]{Definition}
\newtheorem{rem}{Remark}
\newtheorem{ex}{Example}

\DeclareMathOperator{\id}{id}

\newcommand{\bd}[1]{\mathbf{#1}}  
\newcommand\QQ{{\mathbb Q}}
\newcommand{\ZZ}{\mathbb{Z}}      

\newcommand\NN{{\mathbb N}}
\newcommand\FF{{\mathbb F}}
\newcommand\GFq{{\bd{G}(\FF_q)}}

\newcommand{\F}{\mathbb{F}}

\newcommand{\G}{\mathbf{G}}
\newcommand{\sln}[1]{\operatorname{\mathbf{SL}}_{#1}}
\newcommand{\gln}[1]{\operatorname{\mathbf{GL}}_{#1}}
\newcommand{\spn}[1]{\operatorname{\mathbf{Sp}}_{#1}}
\newcommand{\pgln}[1]{\operatorname{\mathbf{PGL}}_{#1}}
\newcommand{\gal}{\operatorname{Gal}}
\newcommand{\im}{\operatorname{im}}
\newcommand{\tensor}{\otimes}
\newcommand{\inv}{^{-1}}

\begin{document}

\nocite{*}

\title{Generic Extensions and Generic Polynomials for Linear Algebraic Groups}

\author{Eric Chen}
\address[Eric Chen]{University of California, Berkeley}
\email{a5584266@berkeley.edu}

\author{J.T. Ferrara}
\address[J.T. Ferrara]{Bucknell University}
\email{jtf019@bucknell.edu}

\author{Liam Mazurowski}
\address[Liam Mazurowski]{Carnegie Mellon University}
\email{liammazurowski@gmail.com}

\thanks{This
research was conducted at the 2015 Louisiana State University Research
Experience for Undergraduates site supported by the National Science
Foundation REU Grant DMS 0648064.}

\begin{abstract}
We construct generic extensions and polynomials for certain linear algebraic groups over finite fields.  In particular, we show the existence of generic polynomials for the cyclic group $C_{2^m}$ in characteristic $p$ for all odd primes $p$.   
\end{abstract}

\maketitle

\section{Introduction}
\hspace{0.25in}A fundamental aspect of the Inverse Galois Problem is describing all extensions of a base field $K$ with a given Galois group $G$. A constructive approach to this problem involves the theory of generic polynomials. For a finite group group $G$, a polynomial $f(t_1,\ldots,t_n; Y) \in K(t_1,\ldots,t_n)[Y]$ is $G$-generic if $\gal(f/K(t_1,\ldots,t_n)) \cong G$ and if for any Galois $G$-extension $M/L$ with $L \supset K$, the parameters $t_1,\ldots,t_n$ can be specialized to $L$ such that $f$ has splitting field $M/L$. A related notion (due to Saltman \cite{Saltman}) is that of generic extensions, which similarly parametrize all Galois $G$-extensions containing $K$. The theory of Frobenius modules, developed by Matzat \cite{Matzat}, is of great practical use for these constructions in positive characteristic.

In our work, we use Frobenius modules to show the existence of and explicitly construct generic polynomials and extensions for various groups over fields of positive characteristic. The methods we develop apply to a broad class of connected linear algebraic groups satisfying certain conditions on cohomology. In particular, we use our techniques to study constructions for unipotent groups, certain algebraic tori and solvable groups, and some split semisimple groups.

An attractive consequence of our work is the construction of generic extensions and polynomials in the optimal number of parameters for all cyclic 2-groups over fields of odd positive characteristic. This constrasts with a theorem of Lenstra, which states that no cyclic 2-group of order $\geq 8$ has a generic polynomial over $\QQ$.

\section{Frobenius Modules}
\label{sect:F-Modules}

The theory of Frobenius modules is intimately related to the construction of generic polynomials. In this section we recall some basic results and definitions, most of which can be found in \cite{Matzat} and \cite{Morales-Sanchez}.

\subsection{Preliminaries}

Let $K \supset \FF_q$ be a field and let $\overline{K}$ be the algebraic closure of $K$.

\begin{defn} \label{def:Frobenius-Module} A \textit{Frobenius module} over $K$ is a pair $(M,\Phi)$ consisting of a finite dimensional vector space $M$ over $K$ and an $\FF_q$-linear map $\Phi: M \rightarrow M$ satisfying
\begin{enumerate}
\item $\Phi(ax) = a^q\Phi(x)$ for $a \in K$ and $x \in M$.
\item The natural extension of $\Phi$ to $\overline{K} \otimes_K M \rightarrow \overline{K} \otimes_K M$ is injective.
\end{enumerate}
\end{defn} 

\begin{defn} The \textit{solution space} $\mathrm{Sol}_\Phi(M)$ of $(M,\Phi)$ is the set of fixed points of $\Phi$, i.e.
$$\mathrm{Sol}_\Phi(M) = \{x \in \overline{K} \otimes_K M : \Phi(x) = x\}.$$
\end{defn}
\begin{rem}
The solution space is clearly an $\FF_q$-subspace of $M$.
\end{rem}

If we pick a $K$-basis $\{e_1, e_2, \dots, e_n\}$ for M, then $\Phi$ can be expressed as an element of $M_n(K)$. Specifically, write
$$\Phi(e_j) = \sum_{i=1}^n a_{ij}e_i$$
where $a_{ij} \in K$ and let $A = (a_{ij}) \in M_n(K)$. We can identify $M$ with $K^n$ with this choice of basis. Then for all $X \in K^n$, we have
$$\Phi(X) = AX^{(q)}$$
where $X = \begin{bmatrix} x_1 & \hdots & x_n\end{bmatrix}^T$ and $X^{(q) }= \begin{bmatrix} x_1^q & \hdots & x_n^q\end{bmatrix}^T$. Note that condition (2) of Definition~\ref{def:Frobenius-Module} ensures that A is nonsingular. From now on, we write $(K^n, \Phi_A)$ to denote the Frobenius module defined by the matrix $A \in \bd{GL}_n(K)$. 
\newline \newline
With this notation, we can write an equivalent definition for $\mathrm{Sol}_{\Phi_A}(M)$:
$$\mathrm{Sol}_{\Phi_A}(M) = \{ X \in \overline{K}^n : AX^{(q)} = X\}.$$

Now suppose $B$ is a matrix representing $\Phi$ in a different basis, with change of basis matrix $U \in \gln n (K)$. Then it can be verified that $B = U^{-1}AU^{(q)}$, where $U^{(q)}$ is element-wise $q^{\text{th}}$-power. This motivates the following definition. 

\begin{defn} Let $(K^n,\Phi_A)$ and $(K^n,\Phi_B)$ be Frobenius modules. We say they are \textit{Frobenius equivalent} over $K$ if there exists some matrix $U \in \bd{GL}_n(K)$ such that
$$B = U^{-1}AU^{(q)}.$$
\end{defn}

\begin{defn} The \textit{splitting field} $E$ of a Frobenius module $(K^n,\Phi_A)$ is the smallest subfield of $\overline{K}$ that contains all the solutions to the system of polynomial equations $AX^{(q)} = X$.
\end{defn}
\begin{rem}It can be verified that $E/K$ is a Galois extension \cite{Morales-Sanchez}. Additionally, Frobenius modules that are equivalent over $K$ define the same splitting field.
\end{rem}

\begin{defn} The \textit{Galois group of a Frobenius module}, is $\gal(K^n, \Phi_A) := \gal(E/K)$, where $E$ is the splitting field of the Frobenius module $(K^n, \Phi_A)$.
\end{defn}



\subsection{Determination of the Galois group}
\label{subsect:Gal-Grp}
We state here the Lang-Steinberg theorem (see \cite[Theorem
1]{Lang:1956ys} and \cite[Theorem 10.1]{Steinberg:1968aa}), which is central to the theory of Frobenius modules. 
\begin{thm}[Lang-Steinberg] \label{thm:L-S} Let $\G \subset \bd{GL}_n$ be a closed connected algebraic subgroup defined over $\FF_q$ and let $A \in \G(K)$. Then there exists $U \in \G(\overline{K})$ such that $U(U^{(q)})^{-1} = A$.
\end{thm}

\begin{rem}In fact, by the results in \cite{Morales-Sanchez}
, the element $U$ given in Theorem~\ref{thm:L-S} lies in $\G(K_{sep})$. Moreover, this implies the splitting field of the Frobenius module defined by $A$ is $K(U) := K(u_{ij})$. 
\end{rem}

Next, we state two theorems due to Matzat \cite{Matzat} that plays an important role in determining the Galois group of a Frobenius module. 

\begin{thm}[Upper Bound Theorem] \label{thm:Upper-Bd} Let $\G \subset \bd{GL}_n$ be a closed connected algebraic subgroup defined over $\FF_q$ and let $A \in \G(K)$. Let $E/K$ be the splitting field of the Frobenius module $(K^n, \Phi_A)$ defined by $A$ and let $U \in \G(E)$ be an element given by the Lang-Steinberg Theorem. Then the map
$$\gal(E/K) \stackrel{\rho}{\rightarrow} \G(\FF_q)$$
$$\sigma \mapsto U^{-1}\sigma(U)$$
is an injective group homomorphism.
\end{thm}

\begin{thm}[Lower Bound Theorem] \label{thm:Lower-Bd} Let $K = \FF_q(\bd{t})$ where $\bd{t} = (t_1, \dots, t_m)$ are indeterminates. Let $\G \subset \bd{GL}_n$ be a closed connected algebraic subgroup defined over $\FF_q$ and let $A \in \G(K)$. Let $\rho: \gal(E/K) \rightarrow \G(\overline{\FF_q})$ be the homomorphism in Theorem 2.6. Then every specialization of $A$ in $\FF_q$ is conjugate in $\G(\overline{\FF_q})$ to an element of $\im(\rho)$.
\end{thm}
We use the special case of the Lower Bound Theorem as stated in the Albert-Maier paper \cite[Theorem
3.4]{Albert:2011vn}. 

\section{Generic polynomials and generic extensions}
\label{sect:Gen-Ext}
First we recall the following definitions, both of which can be found in \cite{Jensen-Ledet-Yui}.
\begin{defn} Let $K$ be a field and $G$ a finite group. A Galois ring extension $S/R$ with group $G$ is called a \textit{generic }$G$\textit{-extension} over $K$, if
\begin{enumerate}
\item $R$ is a localized polynomial ring, i.e.\ $R = K[t_1, \ldots, t_n, \frac{1}{g}]$ where $g$ is a nonzero polynomial, and
\item whenever $L$ is an extension field of $K$ and $M/L$ a Galois $G$-algebra, there is a homomorphism of $K$-algebras $\varphi: R \rightarrow L$, such that $S\otimes_\varphi L$ and $M$ are isomorphic as $G$-algebras over $L$.
\end{enumerate}  
\end{defn}

\begin{defn}A monic separable polynomial $f(Y;\bd{t}) \in K(\bd{t})[Y]$ is called $G$\textit{-generic} over $K$ if the following conditions are satisfied:
\begin{enumerate}
\item $\gal(f(Y;\bd{t})/K(\bd{t})) \cong G$
\item Every Galois $G$-extension $M/L$, where $L$ is a field containing $K$, is the splitting field of a specialization $f(Y;\bd{\xi})$ for some $\bd{\xi} \in L^m$
\end{enumerate}
\end{defn}

Under certain hypotheses it is possible to construct a generic polynomial from a Frobenius module, as detailed in \cite{Morales-Sanchez}. We outline these important results in the following theorem.

\begin{thm}
\label{thm:M-S}
Let $\bd{G} \subset \bd{GL}_n$ be a closed connected algebraic group defined over $\FF_q$. If the following conditions are satisfied
\begin{enumerate} 
\item $H^1(\mathbf{G}, K) = \{1\}$, for all $K \supset \FF_q$
\item Conjugacy in $\GFq$ is equivalent to conjugacy in $\bd{G}(\overline{\FF_q})$; i.e., if $a,b\in \GFq$, and there exists $u \in \bd{G}(\overline{\FF_q})$ such that $uau^{-1} = b$, then there exists $w \in \GFq$ such that $waw^{-1} = b$
\item There exists an algebraic map $A: \mathbb{A}^d \rightarrow \bd{G}$ defined over $\FF_q$ with
\begin{enumerate}
\item $\{A(\xi) : \xi \in \FF_q^d\} \subset \GFq$ is a set of strong generators (for the definition of strong generators, see \cite{Albert-Maier} Section 3.4)
\item if $B \in \bd{GL}_n(L)$, and $(L^n, \Phi_B)$ has Galois group $\GFq$, then $B$ is Frobenius equivalent to $A(\xi)$, for some $\xi \in L^d$
\end{enumerate}
\end{enumerate}
then $\G(\FF_q)$ admits a generic polynomial in $d$ parameters over $\FF_q$. 
\end{thm}
\begin{proof} We sketch the proof of this theorem.
Let $\bd{t} = (t_1, \dots, t_d)$ be indeterminates, then it follows from Theorems~\ref{thm:L-S} and \ref{thm:Upper-Bd} that we get an injective representation

$$R: \gal(A(\bd{t})/\FF_q(\bd{t})) \hookrightarrow \GFq$$
In fact, using conditions (2), (3a), and Theorem \ref{thm:Lower-Bd}, one can show that $R$ is an isomorphism.

Now suppose $M/L$ is a Galois extension with group $\GFq$. Choose an isomorphism $\rho: \gal(M/L) \to \GFq \subset \bd{G}(M)$. By condition (1), there exists some $U \in \bd{G}(M)$ such that the cocycle $\rho: \gal(M/L) \to \bd{G}(M)$ is given by $\rho(\sigma) = U^{-1}\sigma(U)$, for all $\sigma \in \gal(M/L)$. If we let $B = U(U^{(q)})^{-1}$, then $B$ is in $\bd{G}(L)$, and $(L^n, \Phi_B)$ has splitting field $M$. By condition (3b), $B$ is Frobenius equivalent to $A(\xi)$, for some $\xi \in L^d$. 
\par Since $A(\bd{t})$ defines a Frobenius module over an infinite field, this module admits a cyclic basis \cite{Matzat}.
The generic polynomial is then easily calculated to be of the form 
$$Y^{q^n} - a_0Y - a_1Y^q - \cdots - a_{n-1}Y^{q^{n-1}}$$
where the $\{a_i\}$ are elements in the last column of the matrix representing the module with respect to this cyclic basis \cite[Theorem 2.1]{Morales-Sanchez}. For an explicit construction of the coefficients $a_i$, see Corollary 4.2, 4.3 of \cite{Morales-Sanchez}.  
\end{proof}

\section{Generic extensions for $\GFq$ over $\FF_q$}
\label{sect:Thm}
We now state and prove one of the main results of this paper, which will be used extensively in Sections~\ref{sect:app} and \ref{sect:cyclic} to derive other interesting consequences.

\begin{thm}
\label{thm:Gen-Ext}
Let $\G$ be a closed connected algebraic group defined over $\F_q$. Define the Lang-Steinberg map $\lambda: \G \to \G$ by $X \mapsto X(X^{(q)})^{-1}$. Let $S = \FF_q[\G]$ and $R = \lambda^\ast\FF_q[\G]$, where $\lambda^\ast$ is the induced ring homomorphism. Then we have the following: 
\begin{enumerate}
\item $S/R$ is a Galois extension of rings with group $\GFq$.
\item If $M/L$ is a Galois $\GFq$-algebra with $L \supset \FF_q$ such that $H^1(\gal(M/L), \G(M)) = \{1\}$, then there exists an $\FF_q$-algebra homomorphism $\varphi: R \to L$ such that $L \tensor_\varphi S \cong M$ as $\GFq$-algebras over $L$.

\end{enumerate}
\end{thm}

\begin{proof}
First note that $\lambda$ is an epimorphism of varieties by the Lang-Steinberg Theorem. Hence the induced ring homomorphism $\lambda^\ast$ is injective, and thus defines a proper extension of rings. 
We define a left action of $\GFq$ on $S$ by $gp(X) = p(Xg)$. Let $\Gamma = \GFq$.  Then it is easy to check that $\lambda(X) = \lambda(Y)$ if and only if $X\Gamma = Y\Gamma$ as cosets. Thus, we can identify $\G/\Gamma$ with $\G$ bijectively. Consequently, we have an induced map that identifies $\FF_q[\G]$ with $\FF_q[\G/\Gamma]$. However, the functions in $\FF_q[\G/\Gamma]$ are precisely those that are invariant under the action of $\Gamma$; in other words, $R = \FF_q[\G/\Gamma] = \FF_q[\G]^\Gamma = S^\Gamma$.  

By Hilbert's Nullstellensatz, we know that the maximal ideals of $S$ correspond to ideals of functions that vanish at a certain point $X_0 \in \G(\overline \F_q)$.  Let $\mathfrak{m}_{X_0} = \{p\in S:\, p(X_0) = 0\}$ be a maximal ideal in $S$, and let $g\in \G(\overline \F_q)$.  We need to check that if $g\mathfrak{m}_{X_0} = \mathfrak{m}_{X_0}$ and $gp = p \pmod {\mathfrak{m}_{X_0}}$ for all $p\in S$ then $g = \id$.  Since $\G(\F_q)$ acts on $S$ by $(gp)(X) = p(Xg)$, the condition $gp = p \pmod{\mathfrak{m}_{X_0}}$ for all $p\in S$ is equivalent to $p(X_0g) = p(X_0)$ for all $p\in S$.  This implies that $X_0g = X_0$ and so $g = \id$, as needed.  We have now shown that $S/R$ is a Galois extension of rings with group $\G(\F_q)$.
\newline \newline \indent
To show genericity, let $M/L$ be a Galois algebra with $\gal(M/L) = \GFq$. Choose an isomorphism $\rho:\gal(M/L) \to \G(\F_q)$. Then we have the following:
$$\gal(M/L) \stackrel{\rho}{\cong} \GFq \stackrel{i}{\hookrightarrow} \bd{G}(M).$$
Note that $\rho$ is a cocycle from $\gal(M/L)$ to $\bd{G}(M)$. By assumption, there exists some $U \in \bd{G}(M)$ such that
$$\rho(\sigma) = U^{-1}\sigma(U)  \text{ for all } \sigma \in \gal(M/L).$$
Now let $A = U(U^{-1})^{(q)}$. Then $A$ is fixed by $\GFq$, so $A$ must be in $\bd{G}(L)$. We define the following maps:
\begin{eqnarray*}
&& \varphi_U: S \rightarrow M \text{ by } f \mapsto f(U)\\
&& \varphi_A: R \rightarrow L \text{ by } f \mapsto f(A)\\
\end{eqnarray*}
Then the following diagram commutes: 
\newline

\begin{center}
$\begin{CD}
S @>\varphi_U>> M\\
@AAA @AAA\\
R @>\varphi_A>> L\\
\end{CD}$
\end{center}
It remains to verify that $S\otimes_{\varphi_A}L \cong M$. To this end, we define the map
$$\psi: S\otimes_{\varphi_A}L \rightarrow M$$
$$s \otimes a \mapsto \varphi_U(s)a$$
Then we have the following diagram:
\begin{center}
$\begin{CD}
S\otimes_{\varphi_A}L @>\psi>> M\\
@AAA @AAA\\
R\otimes_{\varphi_A}L @= L
\end{CD}$
\end{center}
Since $\gal(S\otimes_{\varphi_A}L) = \GFq$, $\psi$ is a $\GFq$-algebra homomorphism defined over $L$, it is automatically an isomorphism by Proposition 5.1.1 in \cite{Jensen-Ledet-Yui}.
\end{proof}



In order to use Theorem~\ref{thm:Gen-Ext} to produce generic extensions, we first prove a lemma. 
\begin{lem}
\label{lem:Gen-Ext}
Assume the notation as in Theorem~\ref{thm:Gen-Ext}. Fix some $g\in R$ and let $A \in \G(L)$.  Then there is a matrix $B\in \G(L)$ that is Frobenius equivalent to $A$ and satisfies $g(B) \neq 0$.
\end{lem}

\begin{proof}
Consider the map
\begin{gather*}
\ell:\G\to \G\\
X \mapsto X\inv A X^{(q)}.
\end{gather*}
It follows from the Lang-Steinberg theorem that the induced map
$\ell^*:L[\G]\to L[\G]$
is injective.  In particular, $\ell^*(g)$ is not zero.  Since $L$ is infinite, we can therefore choose $Y\in \G(L)$ such that $g(Y\inv A Y^{(q)}) \neq 0$.  
\end{proof}

\begin{cor}
\label{cor:Gen-Ext}
Assuming the same notation as in Theorem~\ref{thm:Gen-Ext}, suppose 
\begin{enumerate}
\item $R[\frac{1}{g}]$ is a localized polynomial ring for some nonzero polynomial $g$, and
\item $H^1(L,\G) = \{1\}$ for all fields $L$ containing $\FF_q$.
\end{enumerate}
Then the ring extension $S[\frac{1}{\lambda^\ast g}]/R[\frac{1}{g}]$ is $\GFq$-generic.
\end{cor}
\begin{proof}
Recall that the extension 
\begin{center}
\begin{tikzpicture}
\node (1) at (0,0) {$\F_q[\G]$};
\node (2) at (0,1.5) {$\F_q[\G]$};
\draw[right hook->] (1) -> node[right]{$\lambda^*$} (2);
\end{tikzpicture}
\end{center}
is $\G(\F_q)$-Galois.  
Hence localizing at $g$ it follows that the extension  
\begin{center}
\begin{tikzpicture}
\node (1) at (0,0) {$R = \F_q[\G][1/g]$};
\node (2) at (0,1.5) {$S = \F_q[\G][1/\lambda^*g]$};
\draw[right hook->] (1) -> node[right]{$\lambda^*$} (2);
\end{tikzpicture}
\end{center}
is also $\G(\F_q)$-Galois.  Moreover, $R$ is a localized polynomial ring.  This is our candidate for a generic extension. 

Let $L$ be a field extending $\F_q$ and let $M$ be a Galois $\G(\F_q)$-algebra over $L$.   Choose an isomorphism $\rho:\gal(M/L)\to \G(\F_q)$.  Then we can view $\rho$ as a 1-cocycle via 
\[
\gal(M/L) \xrightarrow{\rho} \G(\F_q) \hookrightarrow \G(M).
\]
The assumption on the cohomology of $\G$ guarantees that $\rho$ is trivial, and hence $\rho(\sigma) = U\inv \sigma(U)$ for some $U\in \G(M)$.  Let $A = U(U\inv)^{(q)} \in \G(L)$.  

By Lemma~\ref{lem:Gen-Ext} we can choose $Y \in \G(L)$ so that $B = Y\inv A Y^{(q)}$ satisfies $g(B) \neq 0$.  Define the matrix $V = Y\inv U$ so that $B = V(V\inv)^{(q)}$.  Then $\rho(\sigma) = V\inv \sigma(V)$ and so
the following diagram commutes.
\begin{center}
\begin{tikzpicture}
\node (1) at (0,0) {$\F_q[\G]$};
\node (2) at (0,-1.5) {$\F_q[\G]$};
\node (3) at (4,0) {$M$};
\node (4) at (4,-1.5) {$L$};

\draw[right hook->] (2) -> node[right]{$\lambda^*$} (1);
\draw (3) -> (4);
\draw[->] (1) -> node[below]{$\psi$} node[above]{$f \mapsto f(V)$} (3);
\draw[->] (2) -> node[above]{$\varphi$} node[below]{$f\mapsto f(B)$} (4);
\end{tikzpicture}
\end{center}
Moreover, since $g(B)$ is non-zero, it is possible to extend $\varphi$ to a map $\F_q[\G][1/g] \to L$.  Hence localizing at $g$ yields the following commutative diagram.
\begin{center}
\begin{tikzpicture}
\node (1) at (0,0) {$S = \F_q[\G][\frac{1}{\lambda^*g}]$};
\node (2) at (0,-1.5) {$R = \F_q[\G][\frac{1}{g}]$};
\node (3) at (5,0) {$M$};
\node (4) at (5,-1.5) {$L$};

\draw[right hook->] (2) -> node[right]{$\lambda^*$} (1);
\draw (3) -> (4);
\draw[->] (1) -> node[below]{$\psi$} node[above]{$f \mapsto f(V)$} (3);
\draw[->] (2) -> node[above]{$\varphi$} node[below]{$f\mapsto f(B)$} (4);
\end{tikzpicture}
\end{center}
Now consider the tensor product $S\tensor_\varphi L$.  It is a Galois $\G(\F_q)$-algebra over $R\tensor_\varphi L$ and the following diagram commutes.
\begin{center}
\begin{tikzpicture}
\node (1) at (0,0) {$S \tensor_\varphi L$};
\node (2) at (0,-1.5) {$R\tensor_\varphi L$};
\node (3) at (4.5,0) {$M$};
\node (4) at (4.5,-1.5) {$L$};

\draw (2) -> (1);
\draw (3) -> (4);
\draw[->] (1) ->  node[above]{$s \tensor y \mapsto \psi(s)y$} (3);
\draw[->] (2) -> node[above]{$\cong$} node[below]{$r \tensor x \mapsto \varphi(r)x$} (4);
\end{tikzpicture}
\end{center}
Since the bottom map is an isomorphism of $\G(\F_q)$-algebras over $L$, we obtain automatically that the top map is  an isomorphism as well \cite{Jensen-Ledet-Yui}.  Hence the extension $S[\frac{1}{\lambda^\ast g}]/R[\frac{1}{g}]$ is generic. 
\end{proof}

\begin{rem}
If $\F_q[\G]$ is already a localized polynomial ring, then taking $g = 1$ in Corollary~\ref{cor:Gen-Ext} shows that the ring extension $S/R$ given in Theorem~\ref{thm:Gen-Ext} is generic for $\G(\F_q)$ over $\F_q$.
\end{rem}

We conclude this section with a useful theorem from \cite{Jensen-Ledet-Yui} that allows us to extract generic extensions from a semidirect product. We will use this result extensively in Section \ref{sect:cyclic}. 

\begin{thm} 
\label{thm:dir-factor}
Let $S/R$ be a generic $E$-extension, where $E = N \rtimes G$. Then $S^N/R$ is a generic $G$-extension.
\end{thm}

\begin{proof}
See \cite{Jensen-Ledet-Yui}.
\end{proof}

\section{Applications to Linear Algebraic Groups}
\label{sect:app}
In this section, we give some applications to the theory developed in Section~\ref{sect:Thm}. 

\subsection{Unipotent groups}
Recall the definition of a unipotent group.
\begin{defn}
A \textit{unipotent group} $\G$ is a subgroup of $\gln{n}$ such that for all $g \in \G$, the element $1-g$ is nilpotent.
\end{defn}

Here we consider only closed connected unipotent groups, in order to invoke the Lang-Steinberg theorem and subsequently Corollary~\ref{cor:Gen-Ext}. \\
\par Since it is always possible to find a change of basis such that the elements of a unipotent group are upper-triangular, with 1's on the diagonal, we adopt the following notation:

$$U^n := \left\{ M \in \gln n : M \text{ has the form } \begin{bmatrix}
1 & * & \cdots & * \\ 0 & 1 & \cdots & * \\ \vdots & \ddots & \ddots & * \\ 0 & \cdots & 0 & 1
\end{bmatrix}\right\} \subset \gln n$$

Then every unipotent group is a subgroup of $U^n$ for some $n$.
\\
\par
Note that for a finite group $G$ to be realized as the subgroup of $U^n(K)$, where $K \supset \FF_q$, it is necessary that $G$ is a $p$-group. 

\begin{prop}[Hilbert's Theorem 90 for Unipotent Groups] \label{prop:H90-Uni} Let $U$ be a unipotent group defined over $\FF_q$, and $L$ a field containing $\FF_q$. Then $H^1(L, U)$ is trivial.
\end{prop}


This is Proposition 6 from \cite{Serre}. Since the conditions needed to apply Corollary~\ref{cor:Gen-Ext} are fulfilled by Proposition~\ref{prop:H90-Uni}, we have the following.
\begin{thm} Let $K \supset \FF_q$ be a field, $U$ a closed connected unipotent group, and $\lambda^\ast$ the ring homomorphism induced by the Lang-Steinberg map, as in Theorem~\ref{thm:Gen-Ext}. Let $\FF_q[U]$ be the polynomial ring of regular functions on $U$, then the ring extension $\FF_q[U]/\im(\lambda^\ast)$ is $U(\FF_q)$-generic.
\end{thm}

\subsection{Solvable groups}
 
\begin{defn} A linear algebraic group $G$ is \textit{solvable} if there is a subnormal series $\{G_i\}_{i=0}^n$ of closed subgroups such that
$$\{1\} = G_0 \vartriangleleft G_1 \vartriangleleft \cdots \vartriangleleft G_n = G$$
and $G_{i+1}/G_i$ is abelian for all $i$.
\end{defn}
We recall here an important structural theorems regarding solvable groups which can be found as part (iv) of Theorem 10.6 in \cite{Borel}.
\begin{thm} 
\label{Solvable1}
Let $G$ be a connected solvable group. Then we have the decomposition $G = TU$, where $T$ is a torus, and $U$ is the unipotent radical of $G$. 
\end{thm}
\begin{defn} Let $G$ be a connected solvable group defined over $k$. $G$ is split, or $k$-split, if it has a composition series $G = G_0 \supset G_1 \supset \cdots \supset G_n = \{e\}$ consisting of connected subgroups such that the successive quotients are isomorphic to $\mathbf{G}_a$ or $\mathbf{G}_m$. 
\end{defn}
The following theorem can be found as part (iii) of Theorem 15.4 in \cite{Borel}.
\begin{thm} 
\label{Solvable2}
If $k$ is a perfect field, $G$ splits over $k$ if and only if it is trigonalizable; that is, we can view such groups as upper triangular matrices. 
\end{thm}

It is easily seen that this class of solvable groups with ``split tori" satisfies our cohomology condition. We show this with the following proposition.

\begin{prop}[Hilbert's Theorem 90 for split solvable groups] Let $\G$ be a solvable group defined over $\FF_q$, then $H^1(K, \G) = \{1\}$ for all $K$ containing $\FF_q$.
\end{prop}
\begin{proof}
By Theorems \ref{Solvable1} and \ref{Solvable2}, we have $G=TU$, where $T$ is a split torus and $U$ unipotent group normal in $G$.
Note that we have the following exact sequence
$$\{1\} \to U \to \G \to T \to \{1\}$$
which induces the cohomology sequence
$$\cdots \to H^1(K,U) \to H^1(K,\G) \to H^1(K,T) \to \cdots$$
Since $U$ has trivial cohomology by Proposition~\ref{prop:H90-Uni}, and $H^1(K,T) \cong H^1(K, (\G_m)^n)$ is trivial as well, we have that $H^1(K, \G)$ is trivial. 
\end{proof}

As with unipotent groups, since the conditions needed to invoke Corollary~\ref{cor:Gen-Ext} are satisfied, we have the following.

\begin{thm} Let $K \supset \FF_q$ be a field, $\G$ a split solvable group with maximal torus, and $\lambda^\ast$ the ring homomorphism induced by the Lang-Steinberg map, as in Theorem \ref{thm:Gen-Ext}. The ring extension $\FF_q[\G]/\im(\lambda^\ast)$ is $\GFq$-generic. 
\end{thm}
\begin{proof}
It suffices to show that $\FF_q[\mathbf{G}]$ is a localized polynomial ring. This follows from the fact that $\mathbf{G} \cong \mathbf{G}_a \times \cdots \times \mathbf{G}_a \times \mathbf{G}_m \times \cdots \times \mathbf{G}_m$ as $\FF_q$-varieties \cite{Borel}. 
\end{proof}

\subsection{Split semisimple algebraic groups}

Let $\G$ be a connected linear algebraic group defined over $\F_q$ and   assume that $H^1(L,\G) = 1$ for all fields $L \supset \F_q$.  
Assume further that $\G$ is semisimple and split. We want to produce a $\G(\F_q)$-generic extension over $\F_q$. \par

We will employ methods related to the \textit{Bruhat decomposition} \cite{Humphreys}. Let $B$ be a split Borel subgroup of $\G$ (defined over $\FF_q$) and let $U$ be the corresponding unipotent part.  Let $U^-$ be the opposite unipotent part.  

\begin{prop}
\label{bd}
Define a map 
\begin{gather*}
\pi:U^- \times B \to \G\\
(u,b) \mapsto ub.
\end{gather*}
Then $\pi$ is injective and its image is the non-vanishing set of a function $g \in \F_q[\G]$.  This image is often called the big cell in $\G$.
\end{prop}

\begin{proof} See Ch. 28 in \textit{Linear Algebraic Groups} by Humphreys \cite{Humphreys}.
\end{proof}

Let $\mathcal O$ be the big cell in $\G$ and assume it is the non-vanishing set of $g\in \F_q[\G]$.   We know that $\F_q[U] = \F_q[t_1,\hdots,t_m]$ is a polynomial ring over $K$.   We also know that $B = T\times U$ where $T$ is a maximal torus in $\G$.  Moreover, 
\[
\F_q[T] = K[r_1,\hdots,r_k,(r_1\cdots r_k)\inv]
\]
since $\G$ is split.  Now $\pi:U^- \times B \to \mathcal O$ is an isomorphism and it follows that
\[
\F_q[\mathcal O] = \F_q[t_1,\hdots,t_m,s_1,\hdots,s_m,r_1,\hdots,r_k,(r_1\cdots r_k)\inv]
\]
is a localized polynomial ring.  Since $\mathcal O$ is the non-vanishing set of $g\in \F_q[\G]$, we have that $\F_q[\G][1/g] = \F_q[\mathcal O]$ and thus $\F_q[\G][1/g]$ is also a localized polynomial ring.
An application of Corollary \ref{cor:Gen-Ext} now yields the following.

\begin{thm}
Let $\G$ be a connected linear algebraic subgroup defined over $\F_q$.  Assume that $\G$ is semisimple and split and that $H^1(L,\G) = 1$ for every field $L \supset \F_q$.  Let $\mathcal O$ be the big cell in $\G$ and assume that $\mathcal O$ is the non-vanishing set of a function $g\in \F_q[\G]$.  Then the ring extension
\begin{center}
\begin{tikzpicture}
\node (1) at (0,0) {$S = \F_q[\G][1/\lambda^*g]$};
\node (2) at (0,-1.5) {$R = \F_q[\G][1/g]$};

\draw[right hook->] (2) -> node[right]{$\lambda^*$} (1);
\end{tikzpicture}
\end{center}
obtained from the Lang-Steinberg map is $\G(\F_q)$-generic over $\F_q$.

\end{thm}

\begin{rem}
These hypotheses are satisfied in particular for $\sln n(\F_q)$ and $\spn {2n}(\F_q)$ and thus we have obtained generic extensions for these groups over $\F_q$.  
\end{rem}

In principle, it is possible to produce generic polynomials for both $\sln{n}(\F_q)$ and $\spn{2n}(\F_q)$ over $\F_q$. However, since symplectic groups are generally of large order, the generic polynomials produced using our methods are cumbersome to write down. Instead, for the remainder of this subsection we will focus on generic polynomials for $\sln{n}(\F_q)$ over $\F_q$. \par

Take $K = \F_q(t_1,\hdots,t_n)$ and consider the polynomial 
\[
g_n(Y) = Y^{q^n} + \sum_{i=1}^{n-1} t_iY^{q^i} + (-1)^n Y \in K[Y].
\]
In papers of Albert-Maier \cite{Albert-Maier} and Elkies \cite{Elkies} it is shown that the Galois group of $g_n(Y)$ over $K$ is $\sln n (\F_q)$. We will first show that at least for $n=2$ these polynomials are  not generic. 

We have the following useful criterion for determining whether a 1-parameter generic polynomial is possible from \cite{Jensen-Ledet-Yui}.
\begin{thm}[Jensen-Ledet-Yui]
Let $K$ be a field and let $G$ be a finite group.  Assume there is a one parameter polynomial that is $G$-generic over $K$.  Then there exists an embedding $G \hookrightarrow \pgln{2}(K)$.
\end{thm}

\begin{prop}
Assume $q$ is odd and let $K$ be a field extending $\F_q$.  Then there is no embedding of $\sln 2(\F_q)$ in $\mathbf{PGL}_2(K)$.
\end{prop}

\begin{proof}
There is no loss of generality in assuming that $K$ is algebraically closed.  Suppose for contradiction that there is an embedding $\varphi:\sln 2(\F_q) \hookrightarrow \mathbf{PGL}_2(K)$ and consider the element $[A] = \varphi(-I) \in \operatorname{PGL}_2(K)$.  We have $[A]^2 = [I]$ and hence $A^2 = \lambda I$.  
Replacing $A$ by $A/\sqrt \lambda$ if necessary, we may assume that $A^2 = I$.  Now we must have $A \neq \pm I$ since $\varphi$ is injective.  Hence the only possibility is that
\[
A \sim \begin{pmatrix} 1 & 0\\ 0 & -1\end{pmatrix}.
\]
In this case $A$ is regular and so $Z_{\mathbf{PGL}_2(K)}([A]) \cong (K[A])^*/K^*$.  In particular, the centralizer of $[A]$ is abelian.  Since every matrix in $\sln n(\F_q)$ commutes with $-I$,  it follows that $\varphi(\sln n(\F_q)) \subset Z_{\mathbf{PGL}_2(K)}([A])$.   But this implies that $\sln n(\F_q)$ is abelian, contradiction.
\end{proof}

\begin{cor}
The polynomial $g_2(Y)$ is not generic for $\sln 2(\F_q)$ over $\F_q$ when $q$ is odd. 
\end{cor}

Let $K$ be a field extending $\F_q$ and 
let $(M,\Phi)$ be an $n$-dimensional Frobenius module over $K$.  Fix a basis for $M$ over $K$ and let $A$ be the matrix for $\Phi$ with respect to this basis.   Let $\G$ be a closed connected algebraic group defined over $\F_q$.  

\begin{defn}
Let $G_1,G_2\in \mathbf G(\F_q)$.  We say that $G_1$ and $G_2$ form a pair of strong generators if for every $U_1,U_2\in \gln n(\F_q)$ the elements $U_1G_1U_1\inv$ and $U_2G_2U_2\inv$ generate $\GFq$.
\end{defn}

The following result regarding the existence strong generators for $\sln{n}$ is due to \cite{Albert-Maier}. 
\begin{thm}[Albert-Maier]
\label{sg}
There exists a pair of strong generators $G_1,G_2$ for $\sln n(\F_q)$.  Moreover, we can take $G_1,G_2$ to be regular, meaning that their characteristic polynomials equal their minimal polynomials.
\end{thm}

We will also make use of a slightly modified version of the Lower Bound Theorem, also due to \cite{Albert-Maier}.
\begin{thm}
\label{thm:Lower-Bd-AM}
Take $K = \F_q(t_1,\hdots,t_n)$.  Let $\mathfrak m = (t_1-a_1,\hdots,t_n-a_n)$ be a maximal ideal in $\F_q[t_1,\hdots,t_n]$ and define $R = \F_q[t_1,\hdots,t_n]_{\mathfrak m}$.  Let $\overline{A}$ be the matrix obtained from $A$ via the specialization $t_i \mapsto a_i$. If  $A$ belongs to $\gln n(R)$ then the image of $\rho$ contains an element which is conjugate to $\overline{A}$ in $\mathcal  G(\F_q)$.
\end{thm}

We aim to obtain polynomials which are generic for $\sln n(\F_q)$ over $\F_q$.  Take $K = \F_q(s,t_1,\hdots,t_{n-1})$ and define
\[
A = \begin{pmatrix}
& & & & (-1)^{n+1}s\inv\\
s & & & & -t_1\\
& 1 & & & -t_2\\
& & \ddots & & \vdots \\
& & & 1 & -t_{n-1}
\end{pmatrix} \in \sln n(K).
\]
Let $(M,\Phi) := (K^n,\Phi_A)$ be the Frobenius module determined by $A$.  By Theorem \ref{thm:Upper-Bd} there is a faithful representation $\rho:\gal(M,\Phi) \to \sln n(\F_q)$.

We want to show that this representation is surjective.  By Theorem \ref{sg} there exists a pair of regular strong generators $G_1,G_2$ for $\sln n(\F_q)$.  Assume
\[
\chi(x) = x^n + a_{n-1}x^{n-1} + \hdots + a_1x + (-1)^{n}
\]
is the characteristic polynomial for $G_1$.
Note that the matrix $\overline A$ obtained from $A$ via the specialization 
$
s  \mapsto 1$,  $t_i  \mapsto a_i
$
also has characteristic polynomial $\chi(x)$.  Since $G_1$ and $\overline A$ are regular, this implies that $B$ is conjugate to $\overline A$ in $\gln n(\F_q)$.  Theorem \ref{thm:Lower-Bd-AM} implies that $\overline A$ is conjugate to an element of the image of $\rho$.  Hence the image of $\rho$ contains an element which is $\gln n(\F_q)$ conjugate to $G_1$.  The same argument shows that the image of $\rho$ contains an element which is $\gln n(\F_q)$ conjugate to $G_2$.  It follows that $\rho$ is surjective.

Now consider the solution space of the Frobenius module $(M,\Phi)$.  It can be determined by solving the equation $AX^{(q)} = X$.  In fact, it is equivalent to solve $A^TX = X^{(q)}$ \cite{Morales-Sanchez}, which yields the equations
\begin{gather*}
x_2 = s\inv x_1^q\\
x_3 = x_2^q\\
\vdots\\
x_n = x_{n-1}^q\\
(-1)^{n+1}s\inv x_1 - t_1x_2 - \hdots - t_{n-1}x_n = x_n^q.
\end{gather*}
Setting $x_1 = Y$ gives
\[
(-1)^{n+1} s\inv Y - t_1s\inv Y^q - t_{2}s^{-q}Y^{q^{2}}- \hdots - t_{n-1}s^{-q^{n-2}}Y^{q^{n-1}} = s^{-q^{n-1}}Y^{q^n}.
\]
It follows that the polynomial
\[
f(Y) = Y^{q^n} + \sum_{i=1}^{n-1} t_is^{(q^{n-1}-q^{i-1})}Y^{q^i} + (-1)^{n} s^{(q^{n-1}-1)}Y 
\]
has Galois group $\sln n(\F_q)$.

We claim that actually $f$ is generic for $\sln n(\F_q)$ over $\F_q$.  Let $L$ be a field extending $\F_q$ and let $E/L$ be a Galois extension with group $\Gamma \cong \sln n(\F_q)$. Note in particular that $L$ must be infinite. Choose an isomorphism $\rho:\Gamma \to \sln n(\F_q) \subset \sln n(E)$.  We know that $H^1(\Gamma, \sln n(E))$ is trivial and thus there is some $U\in \sln n(E)$ such that $\rho(\sigma) = U\inv \sigma(U)$ for all $\sigma\in \Gamma$.  Put $B = U(U\inv)^{(q)} \in \sln n(L)$ and let $(N,\Psi) := (L^n, \Psi_B)$ be the Frobenius module determined by $B$.  We claim that $B$ is Frobenius equivalent to some specialization of $A$ in $L$.  Given this we are done, for Frobenius equivalent matrices determine the same splitting field.  

By Theorem \ref{thm:M-S} it is possible to choose a cyclic basis $\{v,\Psi(v),\hdots,\Psi^{n-1}(v)\}$ for the  module $(N,\Psi)$.  Let $U$ be the matrix $[v,\Psi(v),\hdots,\Psi^{n-1}(v)]$ and set $\lambda = (\det U)\inv$.  Let $V = [\lambda v,\Psi(v),\hdots, \Psi^{n-1}(v)]$ be the matrix obtained from $U$ by scaling the first column by $\lambda$.  Then $V\inv BV^{(q)}$ is the matrix for $\Psi$ with respect to the basis $\{\lambda v,\Psi(v),\hdots,\Psi^{n-1}(v)\}$.
We see that
\[
V\inv BV^{(q)} = \begin{pmatrix}
& & & &*\\
\lambda^q & & & & *\\
& 1 & & & *\\
& & \ddots & & \vdots\\
& & & 1 & *
\end{pmatrix}.
\]
Moreover, since $\det V = \det V^{(q)} = 1$, this matrix has determinant one.  Thus the entry in the upper right hand corner must be $(-1)^{n+1}\lambda^{-q}$.  It follows that this matrix can be obtained by suitably specializing $A$.  We have now proven the following.

\begin{thm}
The polynomial
\[
f(Y) = Y^{q^n} + \sum_{i=1}^{n-1} t_is^{(q^{n-1}-q^{i-1})}Y^{q^i} + (-1)^{n} s^{(q^{n-1}-1)}Y
\]
is generic for $\sln n(\F_q)$ over $\F_q$.
\end{thm}

\section{Generic Extensions and Polynomials for Cyclic 2-Groups}
\label{sect:cyclic}
In this section, we explore a particularly attractive consequence of the theory developed in Section \ref{sect:Thm} in showing the existence of $C_{2^m}$-generic extensions and polynomials over all fields of positive characteristics. The case of characteristic $2$ is well-known from the theory of Witt
vectors \cite{Rabinoff}, so we shall concentrate in the case of odd characteristic. To this end, we apply our results in Section \ref{sect:Thm} to certain algebraic tori. 
\subsection{Algebraic tori}
\label{subsect:algtori}
Let $K$ be a field and $K_{\text{sep}}$ be the separable closure of $K$.

\begin{defn}
An algebraic group $T$ defined over $K$ is an \textit{algebraic torus} if 
$$T(K_{\text{sep}}) \cong K_{\text{sep}}^\ast \times \cdots \times K_{\text{sep}}^\ast.$$
\end{defn}

Note that an algebraic torus is a closed connected algebraic subgroup of $\mathbf{GL}_n$, and is abelian. We now consider a class of algebraic tori constructed by a technique more generally known as the Weil restriction of scalars, and show that the construction indeed yields an algebraic torus.

Let $E/K$ be a finite, separable extension with basis $e_1, \ldots, e_n$. For $z \in E$ define $m_z: E \rightarrow E$, $y \mapsto zy$. This map is $K$-linear, so we let $M_z$ be the matrix of $m_z$ in the chosen basis. We then get a monomorphism of $K$-algebras, $R: E \hookrightarrow M_n(K)$, $z \mapsto M_z$, where $n = [E:K]$. Furthermore, if $z = a_1e_1+\cdots+a_ne_n \in E$ then we can explicitly identify the image of $z$ as the matrix

$$
  M_z = \left[ 
    \begin{matrix}
      f_{11}(a_1,\ldots,a_n) & \cdots & f_{1n}(a_1,\ldots,a_n) \\
      \vdots & \ddots & \vdots \\
      f_{n1}(a_1,\ldots,a_n) & \cdots & f_{nn}(a_1,\ldots,a_n)
    \end{matrix}
    \right] \in M_n(K).
$$
where $f_{ij} \in K[x_1,\ldots,x_n]$ is a linear combination of the $a_k$.

We then consider the image $E^\ast$ inside of $M_n(K)$, which is a subgroup of $\mathbf{GL}_n(K)$ and defined by polynomial conditions in $K$. Thus, there is an underlying algebraic subgroup $T \subset \mathbf{GL}_n$ defined over $K$ with $T(K) \cong E^\ast$ and $T(K_{\text{sep}}) \cong (K_{\text{sep}} \otimes_K E)^\ast$. So $T$ is an algebraic torus and is given by

$$
  T = \left\{\left[ 
    \begin{matrix}
      f_{11} & \cdots & f_{1n} \\
      \vdots & \ddots & \vdots \\
      f_{n1} & \cdots & f_{nn}
    \end{matrix}
    \right]: \det \neq 0 \right\} \subset \mathbf{GL}_n.
$$

We denote $T$ as $\mathrm{Res}_{E/K} \mathbf{G}_m$. 

We now verify the conditions needed to invoke the appropriate results from Section \ref{sect:Thm}. 

\begin{thm}
\label{thm:cyclic-cohom} 
Let $T = \mathrm{Res}_{E/K} \mathbf{G}_m$ as above. Then for all Galois extensions $M/L$ with $L \supset K$, the cohomology group $H^1(\gal(M/L), T(M))$ is trivial.
\end{thm}
\begin{proof}
See \cite{Knus-Merkurjev-Rost-Tignol} Lemma 29.6.
\end{proof}

We are now in place to apply the results from Section~\ref{sect:Thm} to obtain generic extensions for cyclic groups of order $2^m$.

\begin{thm}
\label{thm:cyclic-gen-ext2}
Let $p$ be an odd prime and let $n \in \NN$. Then there exists a generic extension for $C_{p^n-1}$ in characteristic $p$.
\end{thm}

\begin{proof}
Let $T = \mathrm{Res}_{\FF_{p^n}/\FF_p} \mathbf{G}_m$. Then by Theorem~\ref{thm:cyclic-cohom}, $H^1(K, T) = \{1\}$ for all $K \supset \FF_p$.  We also have that $\FF_p[T]$ is a localized polynomial ring. We see this by taking $U \in T(\FF_p(x_1,\ldots,x_n))$ to be the matrix obtained from substituting $x_1,\ldots,x_n$ into the $f_{ij}$ in the matrix from Section~\ref{subsect:algtori}, i.e.\ the matrix of the most general form for $T$. Then $\FF_p[T] = \FF_p[U] = \FF_p[x_1,\ldots,x_n,\frac{1}{\det(U)}]$. So by Corollary~\ref{cor:Gen-Ext}, there exists a generic extension for $T(\FF_p) \cong \FF_{p^n}^\ast \cong C_{p^n-1}$.
\end{proof}

In fact, we can describe this generic extension more explicitly by identifying $\im(\lambda^\ast) \subset \FF_p[U]$. The map $\lambda^\ast: \FF_p[U] \to \FF_p[U]$ is determined by the images of $x_1,\ldots,x_n$, or equivalently the entries of $U = (u_{ij})$. We have that $u_{ij} \mapsto (U(U^{-1})^{(q)})_{ij}$. Letting $a_{ij} = \lambda^\ast(u_{ij})$, we have $A = \lambda(U)$ and $\im(\lambda^\ast) = \FF_p[A]$. Furthermore, notice that if $\alpha \in \FF_{p^n}$ is a primitive element and we choose $\{1, \alpha, \ldots, \alpha^{n-1}\}$ as the basis for $\FF_{p^n}/\FF_p$, then by the construction of $T$ in Section~\ref{subsect:algtori}, the first column of $U$ is $[x_1 \ldots x_n]^T$. So $\lambda^\ast(x_i) = (U(U^{-1})^{(q)})_{i1}$ describes $\im(\lambda^\ast)$ and $A$ can be obtained by substituting $\lambda^\ast(x_i)$ for $x_i$ in $U$. Summing up, the extension $\FF_p[U]/\FF_p[A]$ is the $C_{p^n-1}$-generic extension obtained from the above theorem.

\begin{cor}
\label{cor:cyclic-gen-ext}
Let $m \in \NN$ and $p$ an odd prime such that $2^m \mid p^n-1$ and $2^{m+1} \nmid p^n-1$ for some $n \in \NN$. Then there exists a generic extension for $C_{2^m}$ in characteristic $p$. Moreover, this extension is given by $\FF_p[A, U^d]/\FF_p[A]$.
\end{cor}

\begin{proof}
Let $d = \frac{p^n-1}{2^m}$, then we can write $C_{p^n-1} \cong C_{2^m} \times C_{d}$. Theorem~\ref{thm:dir-factor} asserts that $\FF_p[U]^{C_d}/\FF_p[A]$ is a generic $C_{2^m}$-extension. So it only remains to show that $\FF_p[U]^{C_d} = \FF_p[A, U^d]$.

Define $r: T \to T \times T$ by $U \mapsto (\lambda(U),U^d)$. Then $\ker(r) = \ker(\lambda) \cap \ker(U \mapsto U^d) = T(\FF_p) \cap \ker(U \mapsto U^d) = T(\FF_p)[d] \cong C_d$. From this we get an injection $T/C_d \hookrightarrow T \times T$. We can decompose $r$ into $i \circ j$ as follows:

$$T \stackrel{j}{\twoheadrightarrow} T/C_d \stackrel{i}{\hookrightarrow} T \times T$$ 

So we get the induced maps,

$$\FF_p[T \times T] \stackrel{i^\ast}{\twoheadrightarrow} \FF_p[T/C_d] \stackrel{j^\ast}{\hookrightarrow} \FF_p[T]$$ 

with the property $r^\ast = j^\ast \circ i^\ast$. If we identify $\FF_p[T \times T] = \FF_p[X,Y]$ (where $X,Y$ are matrices of the most general form for $T$) and $\FF_p[T] = \FF_p[U]$, then $r^\ast: \FF_p[X,Y] \rightarrow \FF_p[U]$ is given by $X \mapsto \lambda(U) = A$ and $Y \mapsto U^d$. So $\im(r^\ast) = \FF_p[A, U^d]$. We also have that $\im(i^\ast) = \FF_p[T/C_d] = \FF_p[U]^{C_d}$ and $j^\ast$ is the natural inclusion map into $\FF_p[U]$. Therefore $\FF_p[U]^{C_d} = \im(j^\ast \circ i^\ast) = \im(r^\ast) = \FF_p[A, U^d]$.

\end{proof}

\begin{lem}
\label{lem:primes}
Let $m \in \NN$ and $p \not\equiv 1 \pmod{2^m}$ be odd. Then there exists $n \in \NN$ such that $2^m \mid p^n-1$ and $2^{m+1} \nmid p^n-1$ if and only if $p \not\equiv -1 \pmod{2^m}$. Furthermore, we can always take $n = \mathrm{ord}(p)$ in $(\ZZ/2^m\ZZ)^\ast$.
\end{lem}
\begin{proof}
We begin by showing we can always take $n = \mathrm{ord}(p)$ in $(\ZZ/2^m\ZZ)^\ast$. Let $n = \mathrm{ord}(p)$. Then $p^n \equiv 1 \pmod{2^m}$, so $2^m \mid p^n-1$. If we also have that $2^m \mid p^k-1$, then it must be the case that $k = nl$, with $l \in \NN$. So if $p^n \equiv 1 \pmod{2^{m+1}}$, then $p^{nl} \equiv 1 \pmod{2^{m+1}}$. We can now proceed by assuming $n = \mathrm{ord}(p)$ in $(\ZZ/2^m\ZZ)^\ast$.

Let $a \in \NN$ and define $v_2(a)$ to be the highest power of 2 dividing $a$, i.e. $v_2(a) = l$ where $a = 2^lp_1^{l_1} \ldots p_k^{l_k}$ as a product of primes. We claim that if $p$ is odd, $v_2(p^{2^k}-1) = v_2(p^2-1)+k-1$. We prove this claim by induction on $k$. The result is trivial for $k =1$. Now suppose the formula holds for $k \geq 1$. Then factor $p^{2^k+1}-1 = (p^{2^k}+1)(p^{2^k}-1)$ and observe that $v_2(ab) = v_2(a)+v_2(b)$. Thus $v_2(p^{2^k+1}-1) = v_2(p^{2^k}+1)+v_2(p^{2^k}-1)$. Since $p$ is odd, $p^2 \equiv 1 \pmod{8}$ and so $p^{2^k}+1 \equiv 2 \pmod{8}$. Hence $v_2(p^{2^k}+1) = 1$ and by the inductive hypothesis,
$$v_2(p^{2^k+1}-1) = 1+v_2(p^{2^k}-1) = v_2(p^2-1)+k$$
thus establishing the claim.

We can finish the proof with this fact. Since $(\ZZ/2^m\ZZ)^\ast$ is a 2-group and $p \not\equiv 1 \pmod{2^m}$, we know that $n = 2^k$ for some $k \geq 1$. We also have that $p^n \equiv 1 \pmod{2^m}$ and so $2^m \mid p^n-1$. Hence $m \leq v_2(p^{2^k}-1) = v_2(p^2-1)+k-1$ which implies $k \geq m+1-v_2(p^2-1)$. Moreover, $k$ is the smallest positive integer satisfying this inequality. Thus $k = \max\{1, m+1-v_2(p^2-1)\}$.

If $k \neq m+1-v_2(p^2-1)$, it must be that $m+1-v_2(p^2-1) \leq 0$, so $m+1 \leq v_2(p^2-1)$. This occurs if and only if $2^{m+1} \mid p^2-1 = (p+1)(p-1)$, which happens if and only if $p \equiv \pm 1 \pmod{2^m}$. So if $p \equiv -1 \pmod{2^m}$, then $n = 2$ and $2^{m+1} \mid p^n-1$.

On the other hand, if $p \not\equiv -1 \pmod{2^m}$, then $k = m+1-v_2(p^2-1)$. Suppose $2^{m+1} \mid p^n-1$, then
$$m+1 \leq v_2(p^{2^k}-1) = v_2(p^2-1)+k-1 = m$$
a contradiction.
\end{proof}
\begin{rem}
If $p \equiv 1 \pmod{2^m}$, then $\FF_p$ contains a primitive $2^m$ root of unity. So by Kummer's Theorem, the extension $\FF_p[s, 1/s]/\FF_p[t,1/t]$ with $s^n = t$ is a generic extension for $C_{2^m}$ over $\FF_p$.
\end{rem}

\subsection{Generic Polynomials for Cyclic 2-Groups}
\label{sect:polys}

We first deal with the easy case where $p\equiv \pm 1 \pmod{2^m}$ which does not necessitate the machinery developed in the previous section.

\begin{prop} 
\begin{enumerate}
\item If $p\equiv 1 \pmod{2^m}$, then there is a generic
$C_{2^m}$-generic polynomial over $\F_p$ in one parameter.
\item If
$p\equiv -1 \pmod{2^m}$, then there is a generic $C_{2_m}$-generic
polynomial over $\F_p$ in two parameters.
\end{enumerate}  
\end{prop}
\begin{proof} (1) Follows directly from Kummer theory since $\FF_p$ contains all the $2^m$ th -roots of unity \cite{Lang}. 

(2) Since $2^m$ divides $p^2-1$, the group $C_{2^m}$ admits a faithful
representation $C_{2^m}\to GL(V)$, where $V$ is a $2$-dimensional
vector space over $\FF_p$. We have an induced action of $G:=C_{2^m}$
on the symmetric algebra $\FF_p[V]$ and on its field of fractions
$\FF_p(V)$. As shown in page 23 of \cite{Jensen-Ledet-Yui}, the
fixed field $\FF_p(V)^G$ is rational over $\FF_p$ and thus $G$ possesses
a generic polynomial in two parameters by Theorem 3 in \cite{Kemper-Mattig}.
\end{proof}

We now exploit the construction of $T = \mathrm{Res}_{\FF_{p^n}/\FF_p} \mathbf{G}_m$ in Section~\ref{subsect:algtori} to obtain a generic polynomial for $C_{2^m}$ in the minimal number of parameters. We do this using an analogous concept of generic Frobenius modules:

\begin{defn}
Let $G$ be a finite group and $A \in \mathbf{GL}_n(K)$ where $K = \FF_q(t_1, \ldots, t_n)$. Then $(K^n, \Phi_A)$ is a \textit{$G$-generic} Frobenius module over $\FF_q$ if:
\begin{enumerate}
\item $\gal(K^n, \Phi_A) \cong G$.
\item If $M/L$ is a $G$-Galois extension of fields with $L \supset \FF_q$, then $M$ is the splitting field of a specialization $A(\xi)$ for some $\xi \in L^n$.
\end{enumerate}
\end{defn}

Let $K = \FF_p(t_1,\ldots,t_n)$ and $A \in T(K)$ by substituting $t_1,\ldots,t_n$ into the $f_{ij}$ in the matrix from Section~\ref{subsect:algtori}.  We make the following proposition:

\begin{prop}
\label{prop:gen-module}
Suppose $m \in \NN$ and $p$ an odd prime such that $2^m \mid p^n-1$ and $2^{m+1} \nmid p^n-1$ for some $n \in \NN$. Let $d = \frac{p^n-1}{2^m}$. Then the Frobenius module $(K^n, \Phi_{A^d})$ is $C_{2^m}$-generic.
\end{prop}

\begin{proof}
First consider the Frobenius module $(K^n, \Phi_A)$. By the Lang-Steinberg Theorem there exists $U \in T(\overline{K})$ such that $A = U(U^{(p)})^{-1}$, so $E = K(U)$ is the splitting field of $(K^n, \Phi_A)$. Observe that
$$A^d(U^d)^{(p)} = (AU^{(p)})^d = U^d$$
since $T$ is abelian. Thus $E_1 = K(U^d)$ is the splitting field for $(K^n, \Phi_{A^d})$. In Corollary~\ref{cor:cyclic-gen-ext}, we showed that $\FF_p[A, U^d] = \FF_p[U]^{C_d}$. Taking the field of fractions of these rings it follows that $\FF_p(A, U^d) = \FF_p(U)^{C_d}$. Note that $K = \FF_p(A)$ so $\FF_p(A, U^d) = K(U^d)$ and $\FF_p(U)^{C_d} = E^{C_d}$. Thus $E_1 = E^{C_d}$. From Theorem~\ref{thm:cyclic-gen-ext2}, $\gal(E/K) \cong C_{p^n-1}$, so $\gal(K^n, \Phi_A) \cong C_{p^n-1}/C_d \cong C_{2^m}$ as desired.

Now let $M/L$ be a Galois extension, $L \supset \FF_p$, with Galois group $C_{2^m}$. Since $T(\FF_p) \cong C_{p^n-1}$, we have an injection $C_{2^m} \hookrightarrow T(\FF_p)$. Choose an isomorphism $\rho: \gal(M/L) \rightarrow T(\FF_p) \subset T(M)$. Since $\gcd(d, 2^m) = 1$, $\rho^d$ is also an isomorphism. View $\rho$ as a 1-cocyle in $H^{1}(\gal(M/L), T(M))$, then by Theorem~\ref{thm:cyclic-cohom}, $H^{1}(\gal(M/L), T(M)) = \{1\}$. Hence there exists $U \in T(M)$ such that $\rho(\sigma) = U^{-1}\sigma(U)$ for all $\sigma \in \gal(M/L)$. So $\rho^d(\sigma) = (U^d)^{-1}\sigma(U^d)$ for all $\sigma \in \gal(M/L)$ since $T$ is abelian.

Let $B = U^d((U^d)^{(p)})^{-1}$. Then $B \in T(L)$ and $B(U^d)^{(p)} = U^d$, so $L(U^d)$ is the splitting field of the Frobenius module $(L^n, \Phi_B)$. If $\sigma \in \gal(M/L(U^d))$ then $\rho^d(\sigma) = (U^d)^{-1}\sigma(U^d) = 1$ so $\gal(M/L(U^d)) = \{1\}$ since $\rho^d$ is injective. Thus $M = L(U^d)$. If we define $C = U(U^{(p)})^{-1}$, then $C \in T(L)$. So we can choose $\xi \in L^n$ such that $A(\xi) = C$. Furthermore $A^d(\xi) = A(\xi)^d = C^d = U^d((U^d)^{(p)})^{-1} = B$. Therefore $M$ is the splitting field of $A^d(\xi)$ and $(K^n, \Phi_{A^d})$ is generic for $C_{2^m}$.
\end{proof}

\begin{rem}
A nearly identical argument shows the Frobenius module $(K^n, \Phi_A)$ is $C_{p^n-1}$-generic.
\end{rem}

We can apply the construction in Theorem \ref{thm:M-S} to obtain a generic polynomial from the Frobenius module $(K^n, \Phi_{A^d})$. We demonstrate this method in Section~\ref{sect:Examples}.

As with generic extensions, using this method we obtain a generic polynomial for $C_{2^m}$ over $\FF_p$ for all $p \not\equiv \pm 1 \pmod{2^m}$.

In summary, we have the following theorem:

\begin{thm}
\label{thm:cyclic-gen-poly}
Let $m \in \NN$ and $p$ an odd prime. Let $n$ be the multiplicative order of $p$ modulo $2^m$. Then there exists a generic polynomial for $C_{2^m}$ in $n$ parameters over $\FF_p$. Moreover, the number $n$ is the minimal number of paramters for any $C_{2^m}$-generic polynomial in characteristic $p$. 
\end{thm}
\begin{proof}
It remains only to show the minimality of $n$. On the one hand, it follows from a theorem of Karpenko-Merkurjev \cite{Karpenko-Merkurjev} that the essential dimension of $C_{2^m}$ over $\FF_p$ ($p \neq 2$) is equal to the least dimension of a faithful linear representation of $C_{2^m}$ over $\FF_p$, which is easily seen to be equal to $n$. On the other hand, the generic dimension (i.e., the minimal number of parameters for a generic polynomial) is an upper bound for the essential dimension \cite{Jensen-Ledet-Yui}, so it is in our case also equal to $n$. 
\end{proof}

\section{Examples}
\label{sect:Examples}
In this section, we explicitly compute generic polynomials for selected groups using the theory developed thus far.

Since the case of the full unipotent groups and the full upper triangular groups are dealt with in Kemper and Mattig's work \cite{Kemper-Mattig}. Here we consider a specific unipotent group that is a subgroup of $U^4$. 

\begin{ex}
\textit{The Quaternion Group}

Recall that the Quarternion Group $Q_8$ is a noncommutative 2-group that satisfies the following relations
$$<i,j \text{ } | \text{ } i^2 = j^2 = (ij)^2 = -1>$$
 We can realize $\bd{G} = Q_8$ as a subgroup of $U^4(\FF_2)$ via the representation
 $$\rho: \bd{G}(\FF_2) \to U^4(\FF_2)$$
 $$i \mapsto \begin{bmatrix} 1&1&0&0\\0&1&0&1\\0&0&1&1\\0&0&0&1\end{bmatrix} \text{, } j \mapsto \begin{bmatrix} 1&0&1&0\\0&1&0&0\\0&0&1&1\\0&0&0&1\end{bmatrix} \text{,and } -1 \mapsto \begin{bmatrix} 1&0&0&1\\0&1&0&0\\0&0&1&0\\0&0&0&1\end{bmatrix}$$
It is not hard to see then, that if we let $e_1, e_2, e_3$ be $\rho(i)-I, \rho(j)-I, \rho(-1)-I$, respectively, then as sets, $\bd{G}(\FF_2)$ can be written in the following form
$$ \bd{G}(\FF_2) = \{Q \in U^4 : Q = I + t_1e_1 + t_2e_2 + t_3e_3\}   \eqno{(*)}$$
where the $t_i$'s are in $\FF_2$. Thus, we can identify $\bd{G}$ with $\mathbb{A}^3$ as varieties. \\
By Proposition~\ref{prop:H90-Uni}, $H^1(\bd{G}(\FF_2)) = {1}$. \\
Using notation defined in Theorem 3.2, we have the Lang-Steinberg map:
$$\lambda: \bd{G} \rightarrow \bd{G}$$
$$ (t_1,t_2,t_3) \mapsto (t_1^2 - t_1, t_2^2-t_2, t_3^2 - t_3 -t_1^3-t_2t_1^2 - t_2^3)$$
So the induced ring homomorphism $\lambda^\ast$ has image
$$ R = \{ p(x,y,z) \in \FF_2[\bd{G}] : x =t_1^2 - t_1, y =t_2^2-t_2, z =  t_3^2 - t_3 -t_1^3-t_2t_1^2 - t_2^3 \} $$
for some $t_1, t_2, t_3  \in K \supset \FF_2$. By Theorem~\ref{thm:Gen-Ext}, $\FF_2[\G]/R$ is a $\bd{G}(\FF_2)$-generic extension. \\
\par
Now to obtain a generic polynomial, we recall from $(*)$ that the general form of an element in the unipotent group describing $Q_8$ is 
$$A(t_1,t_2,t_3) =  \begin{bmatrix}1 & t_1 & t_2 & t_3 \\ 0 & 1 & 0 &  t_1\\ 0 & 0 & 1 & t_1+t_2 \\ 0 & 0 & 0 & 1\end{bmatrix}$$ 
Then we have defined a Frobenius module $(K^n, \varphi_A)$. By Theorem ~\ref{thm:M-S}, we can obtain a generic polynomial using the fact that Frobenius modules over an infinite field (in this case $K(t_1,t_2,t_3)$) is cyclic. Here we give a different approach by methods of Gr\"{o}bner bases. 
\\
\par
We can use Mathematica to directly compute the solutions to the Frobenius system $AX^{(2)} = X$. For notational convenience, we substitute $a$ for $t_1$, $b$ for $t_2$, and $c$ for $t_3$ from now on.

And so we get the following generic polynomial 
$$f(Y) = a^8+a^5b+a^4b^2+a^3b^3+a^4b^4+ab^5+b^8+a^2bc+ab^2c+a^2c^2+abc^2+b^2c^2+c^4+(a^2b+ab^2)Y$$ $$+ (a^2+ab+a^2b+b^2+ab^2)Y^2+(1+a^2+ab+b^2)Y^4+Y^8$$

Note that we need the extra restriction on the indeterminates $a,b,c$; i.e., the condition that $ab(a+b) \neq 0$. Since this is an algebraic condition, by Corollary~\ref{cor:Gen-Ext}, it is always possible to avoid the zeros of this polynomial when specializing. 
\end{ex}

We note here that it is always possible to ``parametrize" a unipotent subgroup as we did in $(*)$, but the parametrizing polynomials may not always be linear.

\begin{ex}
\label{ex:c8-f5}
\textit{Constructing a generic polynomial for $C_8$ over $\FF_5$}

Observe that $5 \not\equiv \pm 1 \pmod{8}$ and the order of $5$ in $(\ZZ/8\ZZ)^\ast$ is 2 so we can find a generic polynomial in 2 parameters. Define $K = \FF_5(s,t)$. As in Section~\ref{sect:polys}, we proceed using Frobenius modules.

To obtain the matrix for our Frobenius module, we find the ``general'' matrix form for $T = \mathrm{Res}_{\FF_{25}/\FF_5} \mathbb{G}_m$, following the construction in Section~\ref{subsect:algtori}. Since 2 is a non-square in $\FF_5$, we can choose $\{1, \sqrt{2}\}$ as our basis for $\FF_{25}/\FF_5$. Then if $z = a+b\sqrt{2} \in \FF_{25}$,
$$m_z(1) = a+b\sqrt{2}$$
$$m_z(\sqrt{2}) = 2b+a\sqrt{2}$$
Thus,
$$
  M_z = \left[ 
    \begin{matrix}
      a & 2b \\
      b & a
    \end{matrix}
    \right]
$$

is our general matrix form. For our Frobenius module we choose the matrix $A^3$ (since $\frac{5^2-1}{8} = 3$), where $A \in T(K)$ is obtained by substituting $a \mapsto s$ and $b \mapsto t$ in $M_z$. Then

$$
  A^3 = \left[ 
    \begin{matrix}
      s & 2t \\
      t & s
    \end{matrix}
    \right]^3
    = \left[ 
    \begin{matrix}
      s^3+st^2 & s^2t+4t^3 \\
      3s^2t+2t^3 & s^3+st^2
    \end{matrix}
    \right]
$$

defines the Frobenius module $(K^2, \Phi_{A^3})$, which is $C_8$-generic by Proposition~\ref{prop:gen-module}.

To obtain a generic polynomial, we first pick a cyclic basis for $K^2$, as done in \cite{Morales-Sanchez}. We choose $v = \left[\begin{smallmatrix}1 \\ 0 \end{smallmatrix}\right]$ to be the generator for the cyclic basis. Then our basis is $\{v, A^3v^{(5)}\}$ and we form the change of basis matrix

$$N = \left[v \mid A^3v^{(5)}\right] = \left[ 
    \begin{matrix}
      1 & s^3+st^2 \\
      0 & 3s^2t+2t^3
    \end{matrix}
    \right]$$

Note that $N$ is clearly non-singular. The Frobenius module under this basis is given by the matrix

$$B = N^{-1}A^3N^{(5)} = \left[ 
    \begin{matrix}
      0 & 4s^{14}t^4 +3s^{10}t^8 + s^8t^{10} + s^6t^{12} + 2s^4t^{14} + s^2t^{16}+3t^{18} \\
      1 & s^{15} + s^{11}t^4 + 2s^9t^6 + 2s^7t^8 + 3s^5t^{10} + 2s^3t^{12} + st^{14}
    \end{matrix}
    \right]
$$

We now aim to solve the system $BX^{(p)} = X$, where $X = [x_1 \ x_2]^T$, to obtain a single polynomial which has the same splitting field as the Frobenius module. In \cite{Morales-Sanchez}, it is shown that we can equivalently solve $X^{(p)} = B^TX$. From this we obtain the generic polynomial
\begin{align*}
f(Y) = Y^{25} 	&+ (4s^{15} + 4s^{11}t^4 + 3s^9t^6 + 3s^7t^8 + 2s^5t^{10} + 3s^3t^{12} + 4st^{14})Y^5 \\
			&+ (s^{14}t^4 + 2s^{10}t^8 + 4s^8t^{10} + 4s^6t^{12} + 3s^4t^{14} + 4s^2t^{16} + 2t^{18})Y
\end{align*}

We can further factor this polynomial (using Magma) as follows:

\begin{align*}
f(Y) = 	Y&\left(Y^8 + (3s^3 + 4s^2t + 4st^2 + 2t^3)Y^4 + s^6 + s^5t + 4s^4t^2 +
4s^3t^3 + 4s^2t^4 + 3st^5 +3t^6\right) \cdot \\
		&\left(Y^8 + (3s^3 + s^2t + 4st^2 + 3t^3)Y^4 + s^6 + 4s^5t + 4s^4t^2 +
s^3t^3 + 4s^2t^4 + 2st^5 + 3t^6\right) \cdot \\
		&\left(Y^8 + (4s^3 + 2st^2)Y^4 + s^2t^4 + 3t^6\right)
\end{align*}

Each of the three degree 8 factors is irreducible and generic for $C_8$ over $\FF_5$.
\end{ex}

\bibliographystyle{amsplain}
\bibliography{GenericPolynomialPaper}

\end{document}